\documentclass[12pt,leqno]{amsart}
\usepackage{amsmath}
\usepackage{amssymb}
\allowdisplaybreaks

\def\underset#1#2{{\mathrel{\mathop {{}_{} {#2}}\limits_{{#1}_{}}}}}
\def\upplim_#1{\underset{#1}{\overline\lim}\;}
\def\lowlim_#1{\underset{#1}{\underline\lim}\;}
\newtheorem{corollary}[equation]{Corollary}
\newtheorem{definition}[equation]{Definition}
\newtheorem{claim}[equation]{\indent{\it Claim}\rm }

\newtheorem{lemma}[equation]{Lemma}

\newtheorem{remark}[equation]{\indent\rm {\it Remark}}
\newtheorem{theorem}[equation]{Theorem}

\newcommand{\N}{\mathbb{N}}
\renewcommand{\P}{{\mathbb{P}}}

\numberwithin{equation}{section}

\title[The solutions of a sequence of decomposable form inequalities]{Finiteness criteria for the solutions of a sequence of decomposable form inequalities} 

\author{Si Duc Quang}

\begin{document}

\maketitle 

\begin{abstract}
In this paper, we give a finiteness criterion for the solutions of the sequence of semi-$q$-decomposable form equations and inequalities, where the semi-$q$-decomposable form is factorized into a family of $q$ nonconstant homogeneous polynomials with the distributive constant not exceeding a certain number.
\end{abstract}

\def\thefootnote{\empty}
\footnotetext{
2010 Mathematics Subject Classification:
Primary 11J68; Secondary 11J25, 11J97.\\
\hskip8pt Key words and phrases: Diophantine approximation; subspace theorem; homogeneous polynomial.}

\section{Introduction}
Let $k$ be a number field and let $M_k$ be the set of places. Let $S$ be a subset of $M_k$ which contains all archimedean places of $k$. Denote by $\mathcal O_S$ and $\mathcal O_{S}^*$ the set of all $S$-integers and the set of all $S$-units, respectively. Denote by $H_S({\bf x})$ the $S$-height of the point ${\bf x}\in k^{m+1}$ (see Section 2 for detailed definitions).

Let $F(x_0,\ldots,x_m)$ be a decomposable form of degree $q$ which can be factorized into linear factors over a finite extension $k'$ of $k$. An interesting problem in Diophantine approximation is to consider the solutions of the following decomposable form inequations (see \cite{EV1,Sch77,Sch73,Sch80} for reference):
$$0<\prod_{v \in S}\|F(x_0, \ldots, x_m)\|_{v} \leqslant c H_{S}^{\lambda}(x_0, \ldots, x_m) \quad \text { in } \quad (x_0, \ldots, x_m) \in \mathcal{O}_{S}^{m+1},$$
where $S$ is a finite set of places of $k$ containing the archimedean places of $k$, $c$ and $\lambda$ are two given positive real numbers.
If $(x_0, \ldots, x_m)$ is a solution of the above inequality, then so is $(\eta x_0, \ldots, \eta x_m)$ for every $\eta \in \mathcal{O}_{S}^{*}$. The solution $(\eta x_0, \ldots, \eta x_m)$ is said to be $\mathcal{O}_{S}^{*}$-proportional to $(x_0, \ldots, x_m)$.

In \cite{GR}, Gy\"{o}ry and Ru proved the following result for a sequence of decomposable form inequalities.

\vskip0.2cm
\noindent
{\bf Theorem A} (see \cite[Theorem 2]{GR}). {\it Let $q, m$ be positive integers. Let $c, \lambda$ be real numbers with $c>0, \lambda<q-2 m$, and  $k'$ a finite extension of $k$. For $n=1,2, \ldots$, let $F_{n}({\bf x})=F_{n}(x_{0}, \ldots, x_{m}) \in \mathcal{O}_{S}[{\bf x}]$ denote a decomposable form of degree $q$ which can be factorized into $q$ linear factors over $k'$, and suppose that these factors are in general position for each $n$, which means if we consider the hyperplanes in $\mathbb{P}^{m}$ defined by these linear factors, they are in general position in $\mathbb{P}^{m}$. Then there does not exist an infinite sequence of $\mathcal{O}_{S}^{*}$-nonproportional ${\bf x}_{n} \in \mathcal{O}_{S}^{m+1}, n=1,2, \ldots$, for which
\begin{align*}
0<\prod_{v \in S}\|F_{n}({\bf x}_{n})\|_{v} \leqslant c H_{S}^{\lambda}({\bf x}_{n})\ 
\text{ and }\ h(F_{n})=o(h({\bf x}_{n})) \text { as } n\rightarrow \infty. 
\end{align*}}

In 2020, Ji, Yan and Yu generalized the above result of Gy\"{o}ry and Ru by considering the case where each form $F_n$ is factorized into a product of homogeneous polynomials $Q_{1,n},\ldots,Q_{q,n}$ (see \cite[Theorem 1.2] {JYY}). To prove this result, they used the method of Yan and Yu in the paper \cite{YY}, which involves constructing the Corvaja-Zannier filtration for the vector space of homogeneous polynomials vanishing on the image of the entire curves, in establishing the second main theorem for families of moving hypersurfaces in subgeneral position and nonconstant entire curves (see \cite[Main Theorem]{YY}).

On the other hand, in 2022, by introducing the notion of ``distributive constant'' of an arbitrary family of hypersurfaces with respect to a projective variety the author gave a general Schmidt's subspace theorem for fixed targets (see \cite[Theorem 1.5]{Qpcf}). Later on, in a very recent work \cite{CQT}, Cao, Thin and the author have given the notion of ``distributive constant'' for families of moving hypersurfaces indexed by an infinite index set and generalized our result in \cite{Qpcf} to the case of sequence of targets (see \cite[Theorem 1.6]{CQT}). 

Motivated by these works, in this paper we will use the notion of ``distributive constant'' to improve and generalize Theorem A of  Gy\"{o}ry and Ru and \cite[Theorem 1.2] {JYY} of Ji-Yan-Yu to the case where each $F_n$ is factorized into a product of homogeneous polynomials $Q_{1,n},\ldots,Q_{q,n}$ and the family of  moving hypersurfaces $\mathcal Q=\{Q_{1,n},\ldots,Q_{q,n}\}$ indexed by the set $\mathbb N$ of positive integers with a certain distributive constant. In order to state our result, we recall the following.

Let $\Lambda$ be an infinite index set. Let $\N=\{0,1,2,\ldots\}$ and for a positive integer $d$, we set
$$\mathcal T_d :=\{(i_0,\ldots, i_m)\in\N^{m+1}\ :\ i_0+\cdots +i_m=d\}.$$
We define that a moving hypersurface $Q$ in $\P^n(k)$ of degree $d,$ indexed by $\Lambda$ is a collection of homogeneous polynomials $\{Q(\alpha)\}_{\alpha\in\Lambda}$ of the same degree $d$.
Then, we can write 
$$Q=\sum_{I\in\mathcal T_d}a_I x^I,$$ 
where ${\bf x}^I = x^{i_0}_0\cdots x^{i_m}_m$ for ${\bf x}=(x_0,\ldots, x_m)$ and $I = (i_0,\ldots,i_m)$, and the $a_I$'s are maps (or functions) from $\Lambda$ to $k$ such that for every $\alpha\in\Lambda$, at least one $a_I(\alpha)$ is non-zero. We note that the support $Q(\alpha)^*=\{(x_0:\cdots:x_n)\in \P^n(\bar k)\ | \sum_{I\in\mathcal T_d}a_I(\alpha)x^I=0\}$ of $Q(\alpha)$ is a fixed hypersurface in $\P^n(\bar k)$ for every $\alpha\in\Lambda$.

\vskip0.2cm
\noindent
{\bf Definition B} (see \cite[Definition 1.5]{CQT}). {\it Let $V$ be a subvariety of $\mathbb P^n(k)$ of dimension $m$ and $\mathcal Q=\{Q_{1}, \ldots, Q_{q}\}$ be a family of $q$ moving hypersurfaces of $\mathbb P^n(k)$ indexed by an infinite index set $\Lambda$. Assume that $V$ is not contained in $Q_j(\alpha)^*\; ( j=1,\dots, q)$ for all, but finitely many $\alpha\in\Lambda$. We define the distributive constant of $\mathcal Q$ with respect to $V$ by the smallest number $\Delta_{\mathcal Q,V}$ such that there exists an infinite index subset $A$ of $\Lambda$ with
$$\Delta_{\mathcal Q,V}=\max_{\emptyset\ne\Gamma\subset \{1,\dots, q\},\alpha\in A}\dfrac{\#\Gamma}{m-\dim(\bigcap_{j\in \Gamma}Q_j(\alpha)^*\cap V(\overline k))}.$$}

\vskip0.1cm
Here we note that $\dim\emptyset =-\infty$, by $\sharp\Gamma$ we denote the number of elements of the set $\Gamma$. 
From the definition it follows that $\Delta_{\mathcal Q,V}\ge 1$. Moreover, if $Q_1,\ldots,Q_q\ (q\geqslant \ell+1)$ are in $\ell-$subgeneral position with respect to $V$ then $\Delta_{\mathcal Q,V}\le\ell-m+1$ (see \cite{CQT}). In the case $V=\mathbb P^n(k)$, we just call $\Delta_{\mathcal Q,\mathbb P^n(k)}$ the distributive constant of the family $\mathcal Q$.

Now, let $q$ be a positive integer and $F\left(x_0, \ldots, x_m\right)$ a homogeneous polynomial in $k\left[x_0, \ldots, x_m\right]\ (m \geq 1)$. We say that $F$ is a semi-$q$-decomposable form if $F$ can be factorized into a product of $q$ nonconstant homogeneous polynomials $Q_1, \ldots, Q_q$ over $\bar{k}$.
Our main result is stated as follows.

\begin{theorem}\label{1.1}
Let $\ell,m,q$ be positive integers and $\Delta\ge 1$ a positive constant. Let $k'$ be a finite extension of $k$ and $S \subset$ $M_{k}$ be a finite set containing all archimedean places. For $n=1,2, \ldots$, let $F_{n}({\bf x})=F_{n}(x_{0}, \ldots, x_{m}) \in \mathcal{O}_{S}[{\bf x}]$ be a sequence of semi-$q$-decomposable forms of degree $\ell$. Assume that $F_n=Q_{1,n} \cdots Q_{q,n}$ over $k'$ with $\deg Q_{j,n}=d_j, 1 \leqslant j \leqslant q$, for each $n$. Let $d=\max _{1 \leqslant j \leqslant q}d_j$. Assume that $\ell>d\Delta\left(\frac{m}{2}+1\right)^2$ and the family of $q$ moving hypersurfaces $\{Q_{1,n}, \ldots ,Q_{q,n}\}$ (indexed by $\mathbb N$) has distributive constant not exceeding $\Delta$ for each $n$. Let $c, \lambda$ be real numbers with $c>0, \lambda<\ell-d\Delta\left(\frac{m}{2}+1\right)^{2}$. Then, there does not exist an infinite sequence of $\mathcal{O}_{S}^{*}$-non-proportional $x_{n} \in \mathcal{O}_{S}^{m+1}, n=1,2, \ldots$, for which
\begin{align}\label{1.2}
0<\prod_{v \in S}\|F_{n}({\bf x}_{n})\|_{v} \leqslant c H_{S}^{\lambda}({\bf x}_{n})
\end{align}
and
\begin{align}\label{1.3}
h(F_{n})=o(h({\bf x}_{n})) \text { as } n\rightarrow \infty.
\end{align}
\end{theorem}

From the above theorem, we immediately get the following corollary.
\begin{corollary}\label{1.4}
Let $k$ be a number field and $\{Q_{1}, \ldots, Q_{q}\}$ a family of $q$ homogeneous polynomials in variables $(x_1,\ldots,x_m)$ with coefficients in $k$ and which has the distributive constant not exceeding a positive number $\Delta$. Let $d_{j}:=\deg Q_{j}\ (1\leqslant j\leqslant q)$ and $d:=\max_{1 \leqslant j \leqslant q} d_{j}$ and $F=Q_{1}\cdots Q_{q}$.  Assume that $\deg F>d\Delta(\frac{m}{2}+1)^2$. Then, for every finite set $S$ of places of $k$ containing the archimedean places of $k$, for each positive number $\lambda<\deg F-d\Delta(\frac{m}{2}+1)^2$ and for each constant $c>0$, the inequality
$$0<\prod_{v \in S}\|F(x_{1}, \ldots, x_{m})\|_{v} \leqslant c H_{S}^{\lambda}(x_{1}, \ldots, x_{m}) \quad \text { in } \quad(x_{1}, \ldots, x_{m}) \in \mathcal{O}_{S}^{m}$$
has only finitely many $\mathcal{O}_{S}^{*}$-non-proportional solutions.
\end{corollary}
This corollary is an improvement and also a generalization of Theorem 4.2 in \cite{CRY}.
 
\section{Auxiliary result}

Let $k$ be a number field. Let $M_k$ be the set of places of $k$ and $M^{\infty}_k$ be the set of Archimedean places. For each $v\in M_k$, we choose the normalized absolute value $|\cdot |_v$ such that $|\cdot |_v=|\cdot|$ on $\mathbb Q$ (the standard absolute value) if $v$ is archimedean, and $|p|_v=p^{-1}$ if $v$ is non-archimedean and lies above the rational prime $p$. For each $v\in M_k$, denote by $k_v$ the completion of $k$ with respect to $v$ and set 
\begin{align*}
n_v :&= [k_v :\mathbb Q_v]/[k :\mathbb Q];\\ 
\text{ and }\|x\|_v:& =|x|^{n_v}_v\text{ for }  x\in k^*.
\end{align*}
Let $S$ be a finite subset of $M_k$, which contains $M^{\infty}_k$. An element $x\in k$ is said to be an $S$-integer if $\|x\|_v\leqslant 1$ for each $v\in M_k\setminus S$. We denote by $\mathcal O_S$ the set of all $S$-integers. Each unit of $\mathcal O_{S}$ is called an $S$-unit. The set of all $S$-units forms a multiplicative group which is denoted by $\mathcal O_{S}^*$.

Throughout this paper, the number $m$ is fixed. For ${\bf x} = (x_0 ,\ldots , x_m)\in k^{m+1}$, define
$$\|{\bf x}\|_v :=\max\{\|x_0\|_v,\ldots,\|x_m\|_v\},\ v\in M_k.$$
Let ${\bf x}=(x_0:\cdots :x_m)\in\P^m(k)$. The absolute logarithmic height of ${\bf x}$  is defined by
$$h({\bf x}):=\sum_{v\in M_k}\log \|{\bf x}\|_v.$$
If $x\in k^*$, we define the absolute logarithmic height of $x$ by
$$h(x):=\sum_{v\in M_k}\log^+ \|x\|_v,$$
where $\log^+a=\log\max\{1,a\}.$
For ${\bf x}=(x_{1}, \ldots, x_{m}) \in k^{m}$, we also define the $S$-height as 
$$H_{S}({\bf x})=\prod_{v \in S}\|x\|_{v}$$ 
and the logarithmic $S$-height as $h_{S}({\bf x})=\log H_{S}({\bf x})$. 
If ${\bf x}\in\mathcal{O}_{S}^{m} \setminus\{0\}$, then $H_{S}({\bf x}) \geqslant 1$ and $H_{S}(\alpha {\bf x})=H_{S}({\bf x})$ for all $\alpha \in \mathcal{O}_{S}^{*}$. It is clear that $H_{S}({\bf x}) \geqslant h({\bf x})$ for any ${\bf x}\in \mathcal{O}_{S}^{m} \setminus\{0\}.$

Let $Q=\sum_{I\in\mathcal T_d}a_I{\bf x}^I$ be a homogeneous polynomial of degree $d$ in $k[x_0,\ldots, x_m]$. The support $Q^*$ of $Q$ is a fixed hypersurface in $\P^n(\bar k)$ defined by
$$ Q^*=\{(x_0:\cdots:x_n)\in \P^n(\bar k)\ | \sum_{I\in\mathcal T_d}a_Ix^I=0\},$$
where $\bar k$ is the algebraic closure of the field $k$. The norm of $Q$ under $v\in S$ is defined by
$$\|Q\|_v =\max\{\|a_I\|_v; I\in\mathcal T_d\}.$$ 
For each ${\bf x}=(x_0,\ldots, x_m)$, it is easy to see that
\begin{align*}
\begin{cases}
\|Q({\bf x})\|_v\leqslant \sharp\mathcal T_d\cdot \|Q\|_v\cdot \|{\bf x}\|^d_v &\text{ if }v\in M^\infty_k,\\
\|Q({\bf x})\|_v\leqslant \|Q\|_v\cdot \|{\bf x}\|^d_v &\text{ if }v\in M_k\setminus M^\infty_k,
\end{cases}
\end{align*}
where $\sharp\mathcal T_d$ is the cardinalty of $\mathcal T_d$.
The height of $Q$ is defined by
$$h(Q)=\sum_{v\in M_k}\log \|Q\|_v.$$
For each $v\in M_k$, we define the Weil function $\lambda_{Q,v}$ by
$$\lambda_{Q,v}({\bf x}):=\log\frac{\|{\bf x}\|_v^d\cdot \|Q\|_v}{\|Q({\bf x})\|_v},\ {\bf x}\in\P^m(k)\setminus\{Q=0\}.$$ 

\section{Schmidt's subspace theorem for moving hypersurfaces}
To prove Theorem \ref{1.1}, we need to prove a Schmidt's subspace theorem for arbitrary families of moving hypersurfaces. We first recall the following.


Let $A\subset\Lambda$  be an infinite subset. We consider the set of all pairs $(C,a)$, where $C$ is a subset of $A$ with finite complement and $a: C\to k$ is a map. Denote by $\mathcal R_{A}^{0}$ the set of equivalence classes of all such pairs, where the equivalence relation is defined as follows: $(C_1,a_1) \sim (C_2,a_2)$ if there exists $C\subset C_1\cap C_2$ such that $C$ has finite complement in $A$ and $a_1|_{C}=a_2|_{C}.$ Then $\mathcal R_{A}^{0}$ has an obvious ring structure. Moreover, the base field $k$ can be embedded into $\mathcal R_{A}^{0}$ as constant maps.

We now consider a family $\mathcal Q=\{Q_{1},\ldots, Q_{q}\}$ of $q$ moving hypersurfaces indexed by an infinite set $\Lambda$ given by, for each $j=1,\ldots,q$ and $\alpha \in \Lambda$,
$$Q_j(\alpha)=\sum_{I \in \mathcal{T}_{d_j}}a_{j, I}(\alpha)\mathbf{x}^{I},\text{ where $d_j=\deg Q_j$.}$$
We fix the notation $\mathcal{I}_{d_j}=\{I_{j,1},\ldots,I_{j,n_{d_j}}\}$. We have the following notions.

$\bullet$ The subset $A$ is said to be coherent with respect to $\mathcal{Q}$ if for every polynomial
$P\in k[x_{1,I_{1,1}},\ldots,x_{1,I_{1,n_{d_{1}}}},\ldots,x_{q,I_{q,1}},\ldots,x_{q,I_{q,n_{d_{q}}}}]$, that is homogeneous in $x_{j,I_{j, 1}},\ldots,x_{j,I_{j,n_{d_j}}}$ for each $j=1,\ldots,q$, either
$$P\left(a_{1,I_{1,1}}(\alpha),\ldots,a_{1,I_{1,n_{d_{1}}}}(\alpha),\ldots,a_{q,I_{q, 1}}(\alpha),\ldots,a_{q,I_{q, n_{d_{q}}}}(\alpha)\right)$$
vanishes for all $\alpha\in A$ or it vanishes for only finitely many $\alpha\in A$.

$\bullet$ Let $\mathbf{x}=[x_{0}:\cdots:x_n]:\Lambda\rightarrow\mathbb{P}^{n}(k)$ be a sequence of points. The subset $A\subset\Lambda$ is said to be coherent with respect to $\mathcal{Q}$ and $\mathbf{x}$ if for every polynomial
$$P\in k[x_{1,I_{1,1}},\ldots,x_{1,I_{1,n_{d_1}}},\ldots,x_{q,I_{q,1}},\ldots,x_{q, I_{q,n_{d_q}}},x_0,\ldots,x_n],$$
that is homogeneous in $x_{j,I_{j,1}},\ldots,x_{j,I_{j,n_{d_j}}}$ for each $j=1,\ldots,q$ and homogeneous in $x_0,\ldots,x_n$, either
$$P\left(a_{1,I_{1,1}}(\alpha), \ldots, a_{1,I_{1,n_{d_{1}}}}(\alpha),\ldots,a_{q,I_{q,1}}(\alpha),\ldots, a_{q,I_{q,n_{d_{q}}}}(\alpha),x_{0}(\alpha),\ldots,x_n(\alpha)\right)$$
vanishes for all $\alpha\in A$ or it vanishes for only finitely many $\alpha\in A$.

By the definition, if $A$ is coherent with respect to $\mathcal{Q}$ and if for some $I_{t} \in \mathcal{I}_{d_j}$ we have $a_{j, I_{t}}(\alpha) \neq 0$ for all but finitely many $\alpha \in A$, then there are maps $\frac{a_{j, I_{s}}}{a_{j, I_{t}}}:\alpha\mapsto \frac{a_{j, I_{s}}(\alpha)}{a_{j, I_{t}}(\alpha)}$ defined on the set of $\alpha \in A$ with $a_{j, I_{t}}(\alpha) \neq 0$. Denote by $a_{j, s, t}$ the equivalence class of $\frac{a_{j, I_{s}}}{a_{j, I_{t}}}$ in $\mathcal{R}_{A}^{0}$. We define
$$
\mathcal{R}_{A, \mathcal{Q}}:=\left\{\frac{P\left(\ldots, a_{j, s, t}, \ldots\right)}{Q\left(\ldots, a_{j, s, t}, \ldots\right)}: P, Q \in k\left[\ldots, y_{j, s, t}, \ldots\right]\right\},
$$
which is the field of fractions of the subring of $\mathcal{R}_{A}^{0}$ generated by all $a_{j,s,t}$ over $k$.

Now suppose that $A$ is coherent with respect to $\mathcal{Q}$ and $\mathbf{x}$. If for some $u$ we have $x_{u}(\alpha) \neq 0$ for all but finitely many $\alpha \in A$, then on the set of these $\alpha$ we have maps $\frac{x_v}{x_u}:\alpha\mapsto\frac{x_v(\alpha)}{x_u(\alpha)}.$ For each $v$ denote by $x_{v,u}$ the equivalence class of this map in $\mathcal{R}_{A}^{0}$. Hence, if we set
$$
\mathcal{R}_{A,(\mathcal{Q}, \mathbf{x})}:=\left\{\frac{P\left(\ldots,a_{j,s,t},\ldots, x_{v,u}, \ldots\right)}{Q\left(\ldots,a_{j, s, t},\ldots,x_{v,u}, \ldots\right)}: P, Q \in k\left[\ldots, y_{j,s,t}, \ldots, y_{v,u}, \ldots\right]\right\},
$$
then $\mathcal{R}_{A,(\mathcal{Q}, \mathbf{x})}$ is the field of fractions of the subring of $\mathcal{R}_{A}^{0}$ generated over $k$ by all such elements $a_{j,s,t}$ and $x_{v,u}$. We note that $\mathcal{R}_{A,\mathcal{Q}}\subset \mathcal{R}_{A,(\mathcal{Q},\mathbf{x})}$.

 In this paper, for an element $a=\frac{P\left(\ldots,a_{j,s,t},\ldots\right)}{Q\left(\ldots,a_{j,s,t},\ldots\right)}\in\mathcal{R}_{A,\mathcal{Q}}$, we call the map $\widehat{a}:\alpha\mapsto\frac{P\left(\ldots,\frac{a_{j,I_{s}}}{a_{j,I_{t}}}(\alpha),\ldots\right)}{Q\left(\ldots,\frac{a_{j, I_{s}}}{a_{j, I_{t}}}(\alpha), \ldots\right)}$ a special representative of $a$. For a polynomial $P=\sum_{I}a_{I}\mathbf{x}^{I}\in\mathcal{R}_{A, \mathcal{Q}}\left[x_0,\ldots,x_n\right]$, assume that $\widehat{a}_{I}$ is a special representative of $a_{I}$ for each $I$. We call $\widehat{P}:=\sum_{I}\widehat{a}_{I}\mathbf{x}^{I}$ a special representative of $P$. Note that such $\widehat{P}$ is well defined at all but finitely many $\alpha\in A$.

\begin{definition}\label{2.1} Let $V$ be a subvariety of $\mathbb P^n(k)$ of dimension $m$ and $\mathbf{x}$ a sequence of points in $\mathbb{P}^{n}(k)$ indexed by an infinite set $\Lambda$. Let $\mathcal Q=\{Q_{1}, \ldots, Q_{q}\}$ be a family of $q$ moving hypersurfaces of $\mathbb P^n(k)$ indexed by $\Lambda$ such that $V$ is not contained in $Q_j(\alpha)^*\; ( j=1,\dots, q)$ for all, but finitely many $\alpha\in\Lambda$. 

(i) We say that $\mathbf{x}$ is $V$-linearly nondegenerate with respect to $\mathcal{Q}$ if for each infinite subset $A\subset \Lambda$, coherent with respect to $\mathcal Q$, there is no linear form $L\in\mathcal R_{A,\mathcal Q}[x_0,\dots,x_n]\setminus\mathcal I_{A,\mathcal Q}(V)$ such that for some (and hence all) special representative $\widehat{L}\left(x_{0}(\alpha),\ldots,x_n(\alpha)\right)=0$ for all but finitely many $\alpha\in A$, where $\mathcal I_{A,\mathcal Q}(V)$ is the ideal in $\mathcal R_{A, \mathcal Q}[x_0,\dots,x_n]$ generated by $\mathcal I(V)$ and $\mathcal I(V)$ is the ideal of all homogeneous polynomials in $k[y_0,\ldots,y_n]$ vanishing on $V$.

(ii) We say that $\mathbf{x}$ is $V$-algebraically nondegenerate with respect to $\mathcal Q$ if for each infinite subset $A\subset \Lambda$, coherent with respect to $\mathcal Q$, there is no homogeneous polynomial $P\in \mathcal R_{A,\mathcal Q}[x_0, \dots, x_n] \setminus \mathcal I_{A, \mathcal Q}(V)$  such that for some (and hence all) special  representative $\widehat P$ we have $\widehat P(\alpha)(x_0(\alpha), \dots, x_n(\alpha)) =0$ for all, but finitely many $\alpha\in A$. 
\end{definition}

The following lemma is a slightly reformulation of  Lemma 2.2 in \cite{CQT}.
\begin{lemma}[{cf. \cite[Lemma 2.2]{CQT}}]\label{2.2}
Let $V$ be a subvariety of $\mathbb P^n(k)$ of dimension $m$ and $\mathcal Q=\{Q_{1}, \ldots, Q_{q}\}$ a family of $q$ moving hypersurfaces of $\mathbb P^n(k)$ indexed by an infinite index set $\Lambda$ of the same degree $d\ge 1$. Let $A$ be an infinitely subset of $\Lambda$ which is coherent with respect to $\mathcal Q$. Assume that $\bigcap_{j=1}^{\ell}Q_j(\alpha)^*\cap V(\overline k)=\emptyset$ and
$$ \dim\left(\bigcap_{j=0}^{\ell}Q_j(\alpha)^*\cap V(\overline k)\right)=m-u\; \text{for all}\; t_{u-1}\leqslant s<t_u,\; 1\leqslant u\leqslant m, $$
for all $\alpha\in A$ outside a finite set, where $t_0,t_1,\dots,t_m$ are integers with $1=t_0<t_1<\dots<t_m=\ell.$ Then there exists $m+1$ hypersurfaces $P_0,\dots,P_m$ in $\mathcal R_{A}^{0}[x_0, \dots, x_n]$ of the forms
$$ P_u=\sum_{j=1}^{t_u}c_{uj}Q_j,\; c_{uj}\in k,\; u=0,\dots, m, $$
such that the system of equations
$$ P_{u} (\alpha)(x_0, \dots, x_n)=0, \hspace{1cm} 0\leqslant u \leqslant m$$
has no solution $(x_0,\dots,x_n)$ with $(x_0:\cdots:x_n) \in V(\overline k)$ for all but finitely many $\alpha \in A$.
\end{lemma}

\begin{lemma}[{cf. \cite[Lemma 3.9]{Q4}}]\label{2.3}
Let $t_0,t_1,\dots,t_m$  be $m+1$ integers such that $1=t_0<t_1<\dots<t_m,$ and let $\Delta=\max_{1\leqslant s\leqslant m}\dfrac{t_s-t_0}{s}.$ Then for every  $m$ real numbers $a_0,a_1,\dots,a_{m-1}$ with $a_0\ge a_1\ge \dots\ge a_{m-1}\ge 1,$ we have
$$ a_0^{t_1-t_0}a_1^{t_2-t_1}\cdots a_{m-1}^{t_m-t_{m-1}}\leqslant (a_0a_1\cdots a_{m-1})^{\Delta}.$$
\end{lemma}

We now prove the following Schmidt's subspace theorem for arbitrary sequence of points and families of moving hypersurfaces.
\begin{theorem}\label{2.4} Let $k$ be a number field and $S \subset M_k$ be a finite set, containing all archimedean places. Further, let $\Lambda$ be an infinite set. Let $V$ be a subvariety of dimension $m$ of $\P^n(k)$ and $\mathcal Q=\{Q_1,\dots, Q_q\}$ be a family of moving hypersurfaces in $\P^n(k)$ indexed by $\Lambda$ with distributive constant $\Delta_{\mathcal Q,V}$ with respect to $V$. Let ${\bf x}=[x_0,\ldots,x_n]:\Lambda \to V$ be a sequence of points satisfying
$$h(Q_j(\alpha))=o(h({\bf x}(\alpha))) \text{ for all }\alpha\in\Lambda \text{ and }j=1, \dots,q$$
(i.e., for all $\delta>0$,  $h(Q_j(\alpha))\leq\delta h({\bf x}(\alpha))$ for all, but finitely many, $\alpha\in \Lambda).$
Then, for any $\epsilon >0,$ there exists an infinite index subset $A\subset \Lambda$ such that
\begin{eqnarray*} \sum_{v\in S}\sum_{j=1}^{q}\frac{\lambda_{Q_j(\alpha), v}({\bf x}(\alpha))}{\deg Q_j}\leqslant \left(\Delta_{\mathcal Q,V}\left(\frac{m}{2}+1\right)^2+\epsilon\right)h({\bf x}(\alpha)
)\end{eqnarray*}
holds for all $\alpha \in A.$
\end{theorem}

We noted that, in order to prove Theorem \ref{1.1}, it suffices to establish the above theorem in the case where $V=\P^n(k)$. However, due to the intrinsic interest and the importance of the Schmidt subspace theorem itself, we state and will prove the above theorem in the more general case where $V$ is a projective subvariety of $\P^n(k)$. 

Our proof strategy consists of three steps. In the first step, we apply the hyperpsurface replacement method to estimate the sum of the Weil functions associated with the given $q$ moving hypersurfaces by bounding it with the sum of the Weil functions of new $m+1$ moving hypersurfaces in general position in $V$. Next, we apply the hypersurface replacement method again to reduce the problem to the case of  $\ell +1$ moving hypersurfaces in general position with respect to the variety defined by the ideal generated by all homogeneous polynomials vanishing on the sequence of points ${\bf x}$. Finally, by employing the construction of a new filtration of Corvaja-Zannier type introduced by Ji, Yan, and Yu, together with their arguments, we estimate the sum of the Weil functions for latest $\ell +1$ moving hypersurfaces and obtain the desired inequality of the theorem.

\begin{proof}[Proof of Theorem \ref{2.4}]
Replacing $Q_{j}$ by $Q_{j}^{d / d_{j}}$ if necessary, where $d$ is the least common multiple of $d_{j}$'s, we can assume that $Q_{1}, \ldots, Q_{q}$ have the same degree $d$. Set
$$Q_{j}(\alpha)=\sum_{I \in \mathcal{T}_d} a_{j, I}(\alpha) \mathbf{x}^{I}, \quad j=1, \ldots, q.$$
Take an infinite index subset of $\Lambda$ such that the distributive constant $\Delta_{\mathcal Q,V}$ is defined on this set. 
Then there exists an infinite index subset $A\subset \Lambda$ which is coherent with respect to $\mathcal Q.$ For simplicity, we may assume that the coefficients of all $Q_j$'s belong to $\mathcal R_{A, \mathcal Q}.$  

If there exists $\alpha\in A$ such that $\bigcap_{j=1}^{q}Q_j(\alpha)^*\cap V(\overline k)\ne \emptyset,$ then $\Delta_{\mathcal Q,V}\ge \frac{q}{m}$
and hence the desired inequality of the theorem is clear. Hence, we may assume that $\bigcap_{j=1}^{q}Q_j(\alpha)^*\cap V(\overline k)=\emptyset$ for all $\alpha\in A.$ 

In this proof, for a sequence of points ${\bf x}:\Lambda\rightarrow\P^n(k)$, we denote by $\mathcal C_{\bf x}$ the set of all positive functions $g$ defined over a subset $C$ of $A$ with finite complement such that
$$ \log^{+}(g(\alpha)) =o(h({\bf x}(\alpha))),$$
i.e., for every $\varepsilon >0$, there exists a subset $C_{\varepsilon}\subset C$ with finite complement satisfying $\log^{+}(g(\alpha))\leqslant \varepsilon h({\bf x}(\alpha))$ for all $\alpha \in C_{\varepsilon}.$ Then, we may choose a function $c_v\in \mathcal C_{\bf x}$ such that
$$ \|Q(\alpha)({\bf x(\alpha)})\|_v\leqslant c_v(\alpha)\|{\bf{x}(\alpha)}\|_v^{d}$$
for all $\alpha\in A,v\in S$ and $Q\in\mathcal Q$ (see Lemma 2.1 in \cite{CQT}).

Denote by $\mathcal J$ the set of all permutations $(i_1,\ldots,i_q)$ of $\{1,\ldots,q\}$. For each permutation $J=(i_1,\dots, i_q)\in \mathcal J,$ there exists $m+1$ integers $t_{J,0}, t_{J,1},\dots, t_{J,m}$ with
$1=t_{J,0}<\dots<t_{J,m}=\ell_J\leqslant q$ such that $\bigcap_{j=1}^{\ell_J}Q_{i_j}(\alpha)^*\cap V(\overline k)=\emptyset$ and
$$\dim\left(\bigcap_{j=1}^{s}Q_{i_j}(\alpha)^*\cap V(\overline k)\right)=m-u\; \text{for all}\; t_{J,u-1}\leqslant s<t_{J,u},\ 1\leqslant u\leqslant m.$$
We easily have $\Delta_{\mathcal Q,V}\geqslant\dfrac{t_{J,u}-1}{u}$ for all $1\leqslant u\leqslant m.$ Denote by $P_{J, 0}\dots, P_{J,m}$ the 
moving hypersurfaces obtained in Lemma \ref{2.3} with respect to the family of moving hypersurfaces $\{Q_{i_1}, \dots, Q_{i_q}\},$ which are written by
$$P_{J,u}=\sum_{j=1}^{t_{J,u}}c_{uj}Q_{i_j}, c_{uj}\in k\; \text{for all}\; u=0,\dots, m.$$
Then, there exists a positive function $h_1\in \mathcal C_{\bf x}$ (chosen commonly for all $J$) such that
$$\|P_{J, u}({\bf x}(\alpha))\|_v\leqslant h_1(\alpha)\max_{1\leqslant j\leqslant t_{J,u}}\|Q_{i_j}({\bf x}(\alpha))\|_v$$
for every for every $v\in S,$ for all $\alpha\in A$ outside a finite subset, and for all $u=0,\dots,m.$ 

We fix a space $v\in S$. By passing to an infinite subset of $A$ if necessary, there is a permutation ${J_v}=(j_{v,1},\ldots,j_{v,q})$ of $\{1,\dots,q\}$ such that
$$\|Q_{j_{v,1}}(\alpha)({\bf x}(\alpha))\|_v\leqslant   \|Q_{j_{v,2}}(\alpha)({\bf x}(\alpha))\|_v\leqslant \dots\leqslant \|Q_{j_{v,q}}(\alpha)({\bf x}(\alpha))\|_v$$
for all $\alpha\in A$. Therefore, there exists a positive function $h_2\in \mathcal C_{\bf x}$ such that
$$\|{\bf x}(\alpha)\|_{v}^{d}\leqslant h_2(\alpha)\max_{1\leqslant i\leqslant \ell_{{J_v}}}\|Q_{j_{v,i}}(\alpha)({\bf x}(\alpha))\|_v
\leqslant h_2(\alpha) \|Q_{i_{\ell_{{J_v}}}}(\alpha)({\bf x}(\alpha))\|_v.$$
Then, we get
\begin{align*}
\prod_{i=1}^{q}\dfrac{\|{\bf x}(\alpha)\|_{v}^{d}}{\|Q_{i}(\alpha)({\bf x}(\alpha))\|_v}&\leqslant \Big(\dfrac{\|{\bf x}(\alpha)\|_v^d}{\|Q_{i_{\ell_{{J_v}}}}(\alpha)({\bf x}(\alpha))\|_v}\Big)^{q-\ell_{{J_v}}+1}\prod_{i=1}^{\ell_{{J_v}}-1}\dfrac{\|{\bf x}(\alpha)\|_v^d}{\|Q_{j_{v,i}}(\alpha)({\bf x}(\alpha))\|_v}\\
&\leqslant \dfrac{h_2(\alpha)^{q-\ell_{{J_v}}+1}}{c_v(\alpha)^{\ell_{{J_v}}-1}}\prod_{i=1}^{\ell_{{J_v}}-1}\dfrac{c_v(\alpha)\|{\bf x}(\alpha)\|_{v}^{d}}{\|Q_{j_{v,i}}(\alpha)({\bf x}(\alpha))\|_v}\\
&\leqslant \dfrac{{h_2}(\alpha)^{q-\ell_{{J_v}}+1}}{c_v(\alpha)^{\ell_{{J_v}}-1}}\prod_{i=0}^{m-1}
\Big(\dfrac{c_v(\alpha)\|{\bf x}(\alpha)\|_{v}^{d}}{\|Q_{j_{v,t_{{J_v},i}}}(\alpha)({\bf x}(\alpha))\|_v}\Big)^{j_{v,t_{{J_v},i+1}}-j_{v,t_{{J_v},i}}}\\
\text{[by Lemma \ref{2.3}]}&\le\dfrac{h_2(\alpha)^{q-\ell_{{J_v}}+1}}{c_v(\alpha)^{\ell_{{J_v}}-1}}\prod_{i=0}^{m-1}
\Big(\dfrac{c_v(\alpha)\|{\bf x}(\alpha)\|_{v}^{d}}{\|Q_{j_{v,t_{{J_v},i}}}(\alpha)({\bf x}(\alpha))\|_v}\Big)^{\Delta_{\mathcal Q,V}}\\
&\leqslant \dfrac{h_2(\alpha)^{q-\ell_{{J_v}}+1}h_1(\alpha)^{-m\Delta_{\mathcal Q,V}}}{c_v(\alpha)^{\ell_{{J_v}}-1-m\Delta_{\mathcal Q,V}}}\prod_{i=0}^{m-1}\Big(\dfrac{\|{\bf x}(\alpha)\|_{v}^{d}}{\|P_{J_v,i}(\alpha)({\bf x}(\alpha))\|_v}\Big)^{\Delta_{\mathcal Q,V}}.
\end{align*}
Since $\|P_{J_{v,m}}(\alpha)({\bf x}(\alpha))\|_v\le\sharp\mathcal T_d\cdot\|P_{J_{v,m}}(\alpha)\|_v\cdot\|{\bf x}(\alpha)\|^d_v$, 
the above inequality yields that
\begin{align}\label{2.5}
\log \prod_{j=1}^{q}\dfrac{\|{\bf x}(\alpha)\|_v^{d}}{\|Q_j(\alpha)({\bf x}(\alpha))\|_v}\leqslant \Delta_{\mathcal Q,V}\log \prod_{i=0}^{m}\dfrac{\|{\bf x}(\alpha)\|_v^{d}}{\|P_{J_v,i}(\alpha)({\bf x}(\alpha))\|_v}+o(h({\bf x}(\alpha))).
\end{align}

Since $A$ is coherent with respect to $\mathcal{Q}$ and $\mathbf{x}$, for all $P \in \mathcal{R}_{A, \mathcal{Q}}\left[x_0, \ldots, x_n\right]$, we have that either $\widehat{P}(\mathbf{x}(\alpha))$ vanishes for all but finitely many $\alpha \in A$, or $\widehat{P}(\mathbf{x}(\alpha))$ vanishes for only finitely many $\alpha \in A$.

We denote by $I_{\mathcal{R}_{A, \mathcal{Q}}} \subset \mathcal{R}_{A, \mathcal{Q}}\left[x_0, \ldots, x_n\right]$ the homogeneous ideal generated by all homogeneous polynomials $P \in \mathcal{R}_{A, \mathcal{Q}}\left[x_0, \ldots, x_n\right]$ such that $\widehat{P}(\mathbf{x}(\alpha))=0$ for all but finitely many $\alpha \in A$. 
Since $\mathcal{R}_{A, \mathcal{Q}}\left[x_0, \ldots, x_n\right]$ is a Noetherian ring, $I_{\mathcal{R}_{A, \mathcal{Q}}}$ is finitely generated. Therefore, $I_{\mathcal{R}_{A, \mathcal{Q}}}$ has a finitely generating set $\{P_{1}, \ldots, P_{s}\}$ of $s$ homogeneous polynomials in $\mathcal{R}_{A, \mathcal{Q}}\left[x_0, \ldots, x_n\right]$ (note that if $V=\mathbb P^n(k)$ and ${\bf x}$ is algebraically nondegenerate then $I_{\mathcal{R}_{A, \mathcal{Q}}}=\{0\}$, and in this case the generating set of $I_{\mathcal{R}_{A, \mathcal{Q}}}$ is empty, i.e., $s=0$).

Let $\Omega$ be an algebraically closed extension of $\mathcal{R}_{A, \mathcal{Q}}$ containing $\mathcal{R}_{A,(\mathcal{Q}, \mathbf{x})}$. 
Consider the projective subvariety $W\subset \mathbb{P}^{n}(\Omega)$ (with the base field $\mathcal{R}_{A, \mathcal{Q}}$) defined by $I_{\mathcal{R}_{A, \mathcal{Q}}}$. Let $\deg W=\Delta$ and $\dim W=\ell$. As pointed out in \cite{JYY}, $\ell>0$. Note that $W$ is a subvariety of $V(\Omega)$, then $0<\ell \leqslant \dim V(\Omega)= m$. 

For a positive integer $N$, denote by $\mathcal{R}_{A, \mathcal{Q}}\left[x_0, \ldots, x_n\right]_{N}$ the vector space of homogeneous polynomials of degree $N$. We set 
$$I_{\mathcal{R}_{A, \mathcal{Q}}, N}:=I_{\mathcal{R}_{A, \mathcal{Q}}} \cap \mathcal{R}_{A, \mathcal{Q}}\left[x_0, \ldots, x_n\right]_{N}$$
and 
$$W_{N}:=\mathcal{R}_{A, \mathcal{Q}}\left[x_0, \ldots, x_n\right]_{N} / I_{\mathcal{R}_{A, \mathcal{Q}}, N}.$$
For any $g \in \mathcal{R}_{A, \mathcal{Q}}\left[x_0, \ldots, x_n\right]_{N}$, denote by $[g]$ the equivalence class of $g$ in $W_{N}$. By the theory of Hilbert polynomials, we have the following claim.
\begin{claim}\label{2.6}
There exists a positive integer $N_{0}$ such that
$$M:=\operatorname{dim}_{\mathcal{R}_{A, \mathcal{Q}}} W_{N}=\frac{\Delta N^{\ell}}{\ell!}+O\left(N^{\ell-1}\right)$$
is a polynomial of $N$ for $N \geqslant N_{0}$.
\end{claim}

Let $\widehat{P}_{1}, \ldots, \widehat{P}_{s}$ be the special representatives of $P_{1},\ldots,P_{s}$ respectively. For $\alpha \in A$ such that all $\widehat{P}_{1}, \ldots, \widehat{P}_{s}$ are well defined at $\alpha$, denote by $I(\alpha)$ the homogeneous ideal in $k\left[x_0, \ldots, x_n\right]$ generated by $\widehat{P}_{1}(\alpha), \ldots, \widehat{P}_{s}(\alpha)$ and by $W(\alpha)$ the variety in $\mathbb{P}^{n}(\bar{k})$ defined by $I(\alpha)$. Then we have the following fact.
\begin{claim}[{see \cite[Lemma 2.2]{JYY}}]\label{2.7}
$\operatorname{dim} W(\alpha)=\ell$ for all but finitely many $\alpha \in A$.
\end{claim}
By passing to an infinite subset of $A$, we can assume that for every $v\in S$, there is a permutation $(i^v_1,\ldots,i^v_{m+1})$ of $\{0,1,\ldots,m\}$ such that
$$\|\widehat{D}_{v,1}(\alpha)({\bf x}(\alpha))\|_v\leqslant   \|\widehat{D}_{v,2}(\alpha)({\bf x}(\alpha))\|_v\leqslant \dots\leqslant \|\widehat{D}_{v,m+1}(\alpha)({\bf x}(\alpha))\|_v$$
for every $\alpha\in A$, where $D_{v,i}=P_{J_v,i^v_i}$ and $\widehat{D}_{v,i}=\widehat{P}_{J_v,i^v_i}$. Similar as Lemma \ref{2.3}, we have the following lemma.
\begin{lemma}[{see also \cite[Lemma 2.3]{JYY}}]\label{2.8}
Let $\alpha \in A$ satisfying the following conditions:
\begin{itemize}
\item[(i)] $\widehat{P}_{1}, \ldots, \widehat{P}_{s}, \widehat{D}_{v,1}(\alpha), \ldots, \widehat{D}_{v,m+1}(\alpha)$ are well defined at $\alpha$,
\item[(ii)] $\operatorname{dim} W(a)=\ell$.
\end{itemize}
Then there exist polynomials $\widehat{\widetilde{D}}_{v,1}(\alpha), \widehat{\widetilde{D}}_{v,2}(\alpha), \ldots, \widehat{\widetilde{D}}_{v,\ell+1}(\alpha) \in$ $k\left[x_0, \ldots, x_n\right]$, where $\widehat{\widetilde{D}}_{v,1}(\alpha)=\widehat{D}_{v,1}(\alpha),$ with
$$\widehat{\widetilde{D}}_{v,t}(\alpha)=\sum_{j=1}^{m-\ell+t} c_{tj} \widehat{D}_{v,j}(\alpha), \quad c_{tj} \in k, t\geq 2$$
such that the system of $(s+\ell)$ equations
\begin{align}\label{2.9}
\begin{cases}
&\widehat{P}_{i}(\alpha)\left(x_0, \ldots, x_n\right) =0\ \forall\ 1\leqslant i\leqslant s, \\
&\widehat{\widetilde{D}}_{v,i}(\alpha)\left(x_0, \ldots, x_n\right) =0\  \forall\ 1\leqslant i\leqslant \ell+1
\end{cases}
\end{align}
has no nontrivial solution in $\bar{k}^{n+1}$.
\end{lemma}

Let $\widetilde{D}_{v,1}={D}_{v,1}$ and $\widetilde{D}_{v,t}=\sum_{j=1}^{m-\ell+t} c_{tj}D_{v,j}\in  \mathcal{R}_{A, \mathcal{Q}}\left[x_0, \ldots, x_n\right],\ \forall\ t\geq 2.$ By using the system of resultants of $P_1,\ldots,P_s,\widetilde{D}_{v,1},\ldots,\widetilde{D}_{v,\ell+1}$, from the coherence with respect to $\mathcal Q$ of $A$, we have that the system of equations (\ref{2.9}) has no nontrivial solution in $\bar{k}^{n+1}$ for all but finitely many $\alpha\in A$. Passing to an infinite subset of $A$ again if necessary, we suppose that (\ref{2.9}) has no nontrivial solution in $\bar{k}^{n+1}$ for all $\alpha\in A$.

Of course $\widehat{\widetilde{D}}_{v,i}$ is a special representative of $\widetilde{D}_{v,i}$ for each $1\leqslant i\leqslant \ell+1$.
By the definition, we easily have
$$\left\|\widehat{\widetilde{D}}_{v,t}(\alpha)(\mathbf{x}(\alpha))\right\|_{v}\leqslant C\left\|\widehat{D}_{v,m-\ell+t}(\alpha)(\mathbf{x}(\alpha))\right\|_{v}$$
for $2\leqslant t \leqslant \ell+1$, where $C$ is a positive constant. Hence
$$\log \frac{\|\mathbf{x}(\alpha)\|_{v}^{d}\left\|\widehat{D}_{v,m-\ell+t}(\alpha)\right\|_{v}}{\left\|\widehat{D}_{v,m-\ell+t}(\alpha)(\mathbf{x}(\alpha))\right\|_{v}} \leqslant \log \frac{\|\mathbf{x}(\alpha)\|_{v}^{d}\left\|\widehat{\widetilde{D}}_{v,t}(\alpha)\right\|_{v}}{\left\|\widehat{\widetilde{D}}_{v,t}(\alpha)(\mathbf{x}(\alpha))\right\|_{v}}+\log h'(\alpha)\text { with } h'\in C_{\mathbf{x}}$$
for $2\leqslant t \leqslant \ell+1$. Combining these inequalities with (\ref{2.5}), we easily have
\begin{align}\label{2.10}
\begin{split}
\log \prod_{j=1}^{q}&\frac{\|\mathbf{x}(\alpha)\|_{v}^{d}\left\|Q_{j}(\alpha)\right\|_{v}}{\left\|Q_{j}(\alpha)(\mathbf{x}(\alpha))\right\|_{v}}
\leqslant \Delta_{\mathcal Q,V}\log \prod_{i=1}^{m}\dfrac{\|{\bf x}(\alpha)\|_v^{d}\left\|D_{v,i}(\alpha)\right\|_{v}}{\|D_{v,i}(\alpha)({\bf x}(\alpha))\|_v}+o(h({\bf x}(\alpha)))\\
&\leqslant \Delta_{\mathcal Q,V}(m-\ell+1) \log \prod_{t=1}^{\ell} \frac{\|\mathbf{x}(\alpha)\|_{v}^{d}\left\|\widehat{\widetilde{D}}_{v,t}(\alpha)\right\|_{v}}{\left\|\widehat{\widetilde{D}}_{v,t}(\alpha)(\mathbf{x}(\alpha))\right\|_{v}}+o(h({\bf x}(\alpha))).
\end{split}
\end{align}

Now, for an integer $N$ big enough, divisible by $d$, using the method of Ji, Yan and Yu in \cite{JYY} we will construct a filtration of $W_N$ with respect to $\left\{\widetilde{D}_{v,1}, \ldots, \widetilde{D}_{v,\ell}\right\}$ as follows. 

Arrange by the lexicographic order the $\ell$-tuples ${\bf i}=\left(i_1, \ldots, i_{\ell}\right)$ of nonnegative integers and set $\|\mathbf{i}\|=\sum_j i_j$. For each ${\bf i}=\left(i_1, \ldots, i_{\ell}\right)$, we define
$$W_{v,{\bf i}}=\sum_{{\bf e}=(e_1,\ldots,e_\ell)\geq {\bf i}} \widetilde{D}_{v,1}^{e_1} \cdots \widetilde{D}_{v,\ell}^{e_{\ell}} \cdot \mathcal{R}_{A, \mathcal{Q}}\left[x_0, \ldots, x_n\right]_{N-d\|{\bf e}\|} .$$
Therefore, $W_{v,(0, \ldots, 0)}=\mathcal{R}_{A, \mathcal{Q}}\left[x_0, \ldots, x_n\right]_N$ and $W_{v,{\bf i}} \supset W_{v,{\bf i}'}$ if ${\bf i}'>{\bf i}$. Then $\left\{W_{v,{\bf i}}\right\}$ is a filtration of $\mathcal{R}_{A, \mathcal{Q}}\left[x_0, \ldots, x_n\right]_N$. Set $W_{v,{\bf i}}^*=\left\{[g] \in W_N \mid g \in W_{v,{\bf i}}\right\}$, where $[g]=g \bmod I_{\mathcal{R}_{A, \mathcal{Q}}, N}$. Hence, $\left\{W_{v,{\bf i}}^*\right\}$ is a filtration of $W_N$. Suppose that ${\bf i}'$ follows ${\bf i}$ in the lexicographic order, then
$$\frac{W_{v,{\bf i}}^*}{W_{v,{\bf i}'}^*} \simeq \frac{\mathcal{R}_{A, \mathcal{Q}}\left[x_0, \ldots, x_n\right]_{N-d\|{\bf i}\|}}{I_N^{{\bf i}}} \text { (see \cite[Lemma 3.8]{YY}),}$$
where $I^{{\bf i}}_{N}$ is the subspace of  $\mathcal R_{A,\mathcal Q}[x_{0}, \ldots, x_{n}]_{N-d\|\ell\|}$ consisting of all $\gamma \in \mathcal R_{A,\mathcal Q}[x_{0}, \ldots, x_{n}]_{N-d\|\ell\|}$ such that
$$\widetilde{D}_{v,1}^{i_{1}}\cdots\widetilde{D}_{v,\ell}^{i_{\ell}}\gamma -\sum_{{\bf e}=(e_{1},\ldots,e_{\ell}) > {\bf i}}\widetilde{D}_{v,1}^{e_{1}}\cdots\widetilde{D}_{v,\ell}^{e_{\ell}}\gamma_{{\bf e}}\in I_{\mathcal R_{A,\mathcal Q},N}$$
for some $\gamma_{\bf e}\in\mathcal R_{A,\mathcal Q}[x_0,\ldots,x_n]_{N-d\|{\bf e}\|}$. We set
$$\Delta_{v,N}^{{\bf i}}:=\dim_{\mathcal{R}_{A, \mathcal{Q}}} \frac{W_{v,{\bf i}}^*}{W_{v,{\bf i}'}^*}=\dim_{\mathcal{R}_{A, \mathcal{Q}}} \frac{\mathcal{R}_{A, \mathcal{Q}}\left[x_0, \ldots, x_n\right]_{N-d\|{\bf i}\|}}{I_N^{{\bf i}}}$$
for every ${\bf i}<{\bf i}_0=(N/d,0,\ldots,0)$ and $\Delta_{v,N}^{{\bf i}_0}=1$.
From the inequality (15) in \cite{JYY}, we have
\begin{align}\label{2.11}
\sum_{{\bf i} \in \tau_N} \Delta_{v,N}^{\bf i}i_j=\frac{\Delta N^{\ell+1}}{(\ell+1)!d}+O\left(N^{\ell}\right),\ \forall 1\le j\le \ell,
\end{align}
where $\tau_N$ is the set of all $\ell$-tuples ${\bf i} \in \mathbb{Z}_{\geq 0}^{\ell}$ with $N-d\|{\bf i}\| \geq 0$.

We choose a basis $\mathcal{B}=\left\{\left[\psi^v_1\right], \ldots,\left[\psi^v_M\right]\right\}$ of $W_N$ such that for every ${\bf i}$ with $\dim W^*_{v,\bf i}=u$ then $\left\{\left[\psi^v_{M-u+1}\right], \ldots,\left[\psi^v_M\right]\right\}$ is a basis of $W^*_{v,\bf i}$ . Let $[\psi^v]$ be an element of the basis, which lies in $W_{v,{\bf i}}^* / W_{v,{\bf i}'}^*$, we may write $\psi^v=\widetilde{D}_{v,1}^{i_1} \cdots \widetilde{D}_{v,\ell}^{i_{\ell}} \gamma$, where $\gamma \in \mathcal{R}_{A, \mathcal{Q}}\left[x_0, \ldots, x_n\right]_{N-d\|{\bf i}\|}$. Therefore, for every $1 \leqslant j \leqslant \ell$, from (\ref{2.11}) we have
\begin{align}\label{2.12}
\begin{split}
&\log \prod_{t=1}^M \frac{\|\mathbf{x}(\alpha)\|_v^N\left\|\widehat{\psi}^v_t(\alpha)\right\|_v}{\left\|\widehat{\psi}^v_t(\alpha)(\mathbf{x}(\alpha))\right\|_v}\\
& \geq\left(\frac{\Delta N^{\ell+1}}{(\ell+1)!d}+O\left(N^{\ell}\right)\right)\log \prod_{j=1}^{\ell} \frac{\|\mathbf{x}(\alpha)\|_v^d\left\|\widehat{\widetilde D}_{v,j}(\alpha)\right\|_v}{\left\|\widehat{\widetilde D}_{v,j}(\alpha)(\mathbf{x}(\alpha))\right\|_v}+\log h^{\prime \prime \prime}(\alpha),
\end{split}
\end{align}
with a function $h^{\prime \prime \prime} \in C_{\mathbf{x}}$.

Fix a basis $\left[\phi_1\right], \ldots,\left[\phi_M\right]$ of $W_N$ with $\phi_1, \ldots, \phi_M \in \mathcal{R}_{A, \mathcal{Q}}\left[x_0, \ldots, x_n\right]$, and let
$$\Phi(\alpha):=\left[\widehat{\phi}_1(\alpha)(\mathbf{x}(\alpha)): \cdots: \widehat{\phi}_M(\alpha)(\mathbf{x}(\alpha))\right],$$
which is a map from $A$ into $\mathbb{P}^{M-1}(k).$
Obviously, $h(\Phi(\alpha))=N h(\mathbf{x}(\alpha))+o(h(\mathbf{x}(\alpha)))$.
Moreover, for $1 \leqslant t \leqslant M$, we may write
$$\psi^v_t\left(x_0, \ldots, x_n\right)=L_{v,t}\left(\phi_1\left(x_0, \ldots, x_n\right), \ldots, \phi_M\left(x_0, \ldots, x_n\right)\right),$$
where $L_{v,1}, \ldots, L_{v,M}$ are linear forms with coefficients in $\mathcal{R}_{A, D}$ of $M$ variables which are linearly independent over $\mathcal{R}_{A, \mathcal{Q}}$. By passing to an infinite subset of $A$ if necessary again, we may assume that $L_{v,1}(\alpha), \ldots, L_{v,M}(\alpha)$ are linearly independent over $k$ for all $\alpha \in A$. Then, we have
$$\widehat{\psi}^v_t(\alpha)(\mathbf{x}(\alpha))=\widehat{L}_{v,t}(\alpha)(\Phi(\alpha)) \text { for } t=1, \ldots, M .$$
Since there are only finitely many possible linear forms $L_{v,t}(v\in S, 1 \leqslant t \leqslant M)$, we denote them by $L_1, \ldots, L_\mu$ and let $\mathcal{H}:=\left\{L_1, \ldots,L_\mu\right\}$.

We claim that the sequence of points $\Phi(\alpha)$ of $\mathbb{P}^{M-1}(k)$ indexed by $A$ is linearly nondegenerate with respect to $\mathcal{H}$. Indeed, if not, then there are a infinite coherent subset index set $B \subset A$ and a linear form $L \in \mathcal{R}_{B, \mathcal{H}}\left[x_1, \ldots, x_M\right]$ such that $\widehat{L}(\alpha)(\Phi(\alpha))=0$ for all but finitely many $\alpha \in B$, which contradicts the assumption that there is no homogeneous polynomial $P \in \mathcal{R}_{A, \mathcal{Q}}\left[x_0, \ldots, x_n\right] \backslash I_{\mathcal{R}_{A, \mathcal{Q}}}$ such that $\widehat{P}(\alpha)(\mathbf{x}(\alpha))=0$ for infinitely many $\alpha \in A$ (note that $\mathcal{R}_{A, \mathcal{H}} \subset \mathcal{R}_{A, \mathcal{Q}}$).

Combining (\ref{2.10}) and (\ref{2.12}), we have
\begin{align}\label{2.13}
\begin{split}
\frac{\Delta N^{\ell+1}}{(\ell+1)!d}(1+o(1))&\sum_{v \in S} \log \prod_{j=1}^q \frac{\|\mathbf{x}(\alpha)\|_v^d\left\|Q_j(\alpha)\right\|_v}{\left\|Q_j(\alpha)(\mathbf{x}(\alpha))\right\|_v} \\
\leqslant &\Delta_{\mathcal Q,V}(m-\ell+1)\left[\sum_{v \in S} \max _K \sum_{j \in K} \log \frac{\|\Phi(\alpha)\|_v\left\|\widehat{L}_j(\alpha)\right\|_v}{\left\|\widehat{L}_j(\alpha)(\Phi(\alpha))\right\|_v}\right.\\
&\left.-M \sum_{v \in S} \log \|\Phi(\alpha)\|_v+MN \sum_{v \in S} \log \|\mathbf{x}(\alpha)\|_v\right]+o(h(\mathbf{x}(\alpha))),
\end{split}
\end{align}
where $\max _K$ is taken over all subsets $K$ of $\{1, \ldots, \mu\}$ with $\sharp K=M$ such that $\widehat{L}_j(\alpha)$, $j \in K$, are linearly independent over $k$ for all $\alpha \in A$. 

By \cite[Theorem B4.1.6]{Ru}, we have
\begin{align}\label{2.14}
\sum_{v \in S} \max _K \sum_{j \in K} \log \frac{\|\Phi(\alpha)\|_v\left\|\widehat{L}_j(\alpha)\right\|_v}{\left\|\widehat{L}_j(\alpha)(\Phi(\alpha))\right\|_v} \leq\left(M+\frac{1}{2}\right) h(\Phi(\alpha))
\end{align}
for all $\alpha \in A$. 
Since the above inequality is independent of the choice of components of ${\bf x}(\alpha)$, we can choose the components of ${\bf x}(\alpha)$ being $S$-integers so that
\begin{align}\label{2.15}
\begin{cases}
& \sum_{v\in S}\log ||{\bf x}(\alpha)||_v=h({\bf x}(\alpha))+O(1), \\
&\sum_{v\in S}\log\Vert \Phi(\alpha)\Vert_v\leqslant h(\Phi(\alpha))+O(1) \leqslant Nh({\bf x}(\alpha))+O(1).
\end{cases}
\end{align}
Combining (\ref{2.13}), (\ref{2.14}) and (\ref{2.15}), we obtain
\begin{align*}
\frac{\Delta N^{\ell+1}}{(\ell+1)!d}(1+o(1))& \sum_{v \in S} \log \prod_{j=1}^q \frac{\|\mathbf{x}(\alpha)\|_v^d\left\|Q_j(\alpha)\right\|_v}{\left\|Q_j(\alpha)(\mathbf{x}(\alpha))\right\|_v}\\
&\leq\Delta_{\mathcal Q,V}(m-\ell+1)\left(M+\frac{1}{2}\right) N h(\mathbf{x}(\alpha))+o(h(\mathbf{x}(\alpha))).
\end{align*}
By Claim \ref{2.6}, the above inequality implies that
$$\frac{1}{d} \sum_{v \in S} \log \prod_{j=1}^q \frac{\|\mathbf{x}(\alpha)\|_v^d\left\|Q_j(\alpha)\right\|_v}{\left\|Q_j(\alpha)(\mathbf{x}(\alpha))\right\|_v} \leqslant \Delta_{\mathcal Q,V}(m-\ell+1)(\ell+1+o(1)) h(\mathbf{x}(\alpha)).$$
Finally, we take $N$ large enough so that $\Delta_{\mathcal Q,V}(m-\ell+1)o(1)<\epsilon$ for the given $\epsilon$ in the theorem, and note that $(m-\ell+1)(\ell+1)\leqslant (\frac{m}{2}+1)^2.$ Then from the above inequality, we have
$$\sum_{v \in S}\log \prod_{j=1}^q \frac{\lambda_{Q_j(\alpha),v}({\bf x})}{\deg Q_j} \leq\left(\Delta_{\mathcal Q,V}(\frac{m}{2}+1)^2+\varepsilon\right) h(\mathbf{x}(\alpha))$$
for all but finitely many $\alpha \in A$. Passing to an infinite subset of $A$ if necessary, we may assume that this inequality hold for all $\alpha\in A$.
The theorem is proved.
\end{proof}

\section{Semi-decomposable form equations and inequalities}

\begin{proof}[Proof of Theorem \ref{1.1}]
Let $S'$ be a subset of $M_{k'}$ which consists of the extension of the places of $S$ to $k'$. Then every $S$-integer in $k$ is also an $S'$-integer in $k'$. Moreover, we have $H_{S}({\bf x}_{n})=H_{S'}({\bf x}_{n})$ and
$$\prod_{v \in S}\|F_{n}({\bf x}_{n})\|_{v}=\prod_{w \in S'}\|F_{n}({\bf x}_{n})\|_{w} \text { for } {\bf x}_{n} \in \mathcal{O}_{S}^{m+1}.$$
So the inequality (\ref{1.2}) is preserved when we work on $k'$. Therefore, without loss of generality we may assume that $k'=k$. 

Suppose on the contrary that there is an infinite sequence ${\bf x}_{n}=(x_{0, n}, \ldots, x_{m,n}) \in \mathcal{O}_{S}^{m+1}$ which satisfies (\ref{1.2}) and (\ref{1.3}). If the heights $\{h({\bf x}_{n})\}_{n\geqslant 1}$ are bounded then the set $\{{\bf x}_{n};n\geqslant 1\}$ is actually a finite union of sets of proportional vectors. Therefore, we may suppose that $h({\bf x}_{n})$ are not bounded. 

Without loss of generality, we may assume that $h({\bf x}_{n})\rightarrow \infty$ as $n\rightarrow \infty$. Then, by the assumption, (\ref{1.3}) also holds. Furthermore it follows that $H_{S}({\bf x}_{n})\rightarrow \infty$ as $n\rightarrow \infty$ (since ${\bf x}_n\in\mathcal{O}_{S}^{m+1}$). By $\max _{0\leqslant j\leqslant m}h(Q_{j, n}) \leqslant h(F_{n})+O(1)$, from the inequality (\ref{1.3}) we deduce that 
$$\max _{0\leqslant j\leqslant m}h(Q_{j, n})=o(h({\bf x}_{n})) \text { as } n\rightarrow \infty.$$
Take $\epsilon$ a positive number small enough so that $\ell>d\Delta\left(\frac{m}{2}+1\right)^{2}+\lambda+\epsilon$.
By Theorem \ref{2.4}, there is an infinite subsequence ${\bf x}_{n_p}\in \mathcal{O}_{S}^{m+1},p=1,2, \ldots$, of $\{{\bf x}_{n}\}$, which without loss of generality we may assume to be $\{{\bf x}_{n}\}$ itself, such that
$$\sum_{v \in S} \sum_{j=1}^{q}\dfrac{1}{d_{j,n}} \log \frac{\|{\bf x}_{n}\|_{v}^{d_{j,n}} \cdot\|Q_{j, n}\|_{v}}{\|Q_{j, n}({\bf x}_{n})\|_{v}} \leqslant \left(\Delta\left(\frac{m}{2}+1\right)^{2}+\frac{\epsilon}{d}\right) h({\bf x}_{n}).$$
By the assumption that $d\geqslant \max_{1\leqslant j\leqslant n}d_{j,n}$, the above inequality implies that
\begin{align}\label{3.1}
\sum_{v \in S} \sum_{j=1}^{q}\log \frac{\|{\bf x}_{n}\|_{v}^{d_{j,n}}\cdot\|Q_{j, n}\|_{v}}{\|Q_{j, n}({\bf x}_{n})\|_{v}} \leqslant \left(d\Delta\left(\frac{m}{2}+1\right)^{2}+\epsilon\right) h({\bf x}_{n}).
\end{align}
On the other hand, since ${\bf x}_{n} \in \mathcal{O}_{S}^{m+1}$, we have
\begin{align}\label{3.2}
h({\bf x}_{n}) \leqslant \log H_{S}({\bf x}_{n}).
\end{align}
With the note that $\ell=\sum_{j=0}^md_{j,n}$ for each $n$, from (\ref{3.1}) and (\ref{3.2}) we have 
$$\prod_{v \in S} \frac{\|{\bf x}_{n}\|_{v}^{\ell} \cdot \prod_{j=1}^{q}\|Q_{j, n}\|_{v}}{\|F_{n}({\bf x}_{n})\|_{v}} \leq(H_{S}({\bf x}_{n}))^{d\Delta\left(\frac{m}{2}+1\right)^{2}+\epsilon}$$
whence
$$\frac{H_{S}^{\ell}({\bf x}_{n}) \cdot \prod_{v \in S} \prod_{j=1}^{q}\|Q_{j, n}\|_{v}}{\prod_{v \in S}\|F_{n}({\bf x}_{n})\|_{v}} \leq(H_{S}({\bf x}_{n}))^{d\Delta\left(\frac{m}{2}+1\right)^{2}+\epsilon}.$$
It is clear that $\prod_{v \in S} \prod_{j=1}^{q}\|Q_{j, n}\|_{v} \geq c' \prod_{v \in S}\|F_{n}\|_{v} \geq c'$ for some positive constant $c'$.
Therefore, from (\ref{1.2}), for $n=1,2, \ldots$, we have
\begin{align*}
H_{S}^{\ell}({\bf x}_{n})&\leqslant \frac{1}{c'} \prod_{v \in S}\|F_{n}({\bf x}_{n})\|_{v} \cdot(H_{S}({\bf x}_{n}))^{d\Delta\left(\frac{m}{2}+1\right)^{2}+\epsilon}\\ 
&\leqslant \frac{c}{c'}(H_{S}({\bf x}_{n}))^{d\Delta\left(\frac{m}{2}+1\right)^{2}+\epsilon+\lambda}. 
\end{align*}
Since $H_{S}({\bf x}_{n})\rightarrow \infty$ as $n\rightarrow \infty$, by letting $n\rightarrow \infty$ we get $\ell>d\Delta(\frac{m}{2}+1)^{2}+\lambda+\epsilon$. This is a contradiction. Hence, this completes the proof of Theorem \ref{1.1}.
\end{proof}

\begin{remark}\label{3.3}
{\rm
a) If the family of moving hypersurfaces is in general position, our result will imply \cite[Theorem 1.2] {JYY} of  Ji, Yan and Yu.

b) With the same arguments in \cite{JYY} (see also in \cite{GR}), we will get a result on semi-$q$-decomposable form equations as follows.

Consider the $S$-integer solutions of a sequence of semi-decomposable form equations of the form
\begin{align}\label{3.4}
F_{n}({\bf x})=G_{n}({\bf x}),
\end{align}
where $G_{n}({\bf x})$ are nonzero polynomials.
Let ${\bf x}_{n}$ be the $\mathcal{O}_{S}^{*}$-nonproportional solutions of (\ref{3.4}). There is a positive constant $c$ such that
$$0<\prod_{v \in S}\|F_{n}({\bf x}_{n})\|_{v}=\prod_{v \in S}\|G_{n}({\bf x}_{n})\|_{v} \leqslant c \prod_{v \in S}\|G_{n}\|_{v}(H_{S}({\bf x}_{n}))^{\deg  G_{n}}.$$
Applying Theorem \ref{1.1} for $c$ and $\ell-d\Delta\left(\frac{m}{2}+1\right)^{2}-1<\lambda<\ell-d\Delta(\frac{m}{2}+1)^{2}$, we get the following theorem.}
\end{remark}
\begin{theorem}\label{3.5}
Let $\ell,m,q$ be positive integers. Let $k'$ be a finite extension of $k$ and $S \subset$ $M_{k}$ a finite set containing all archimedean places. For $n=1,2, \ldots$, let $F_{n}({\bf x})=$ $F_{n}(x_{0}, \ldots, x_{m}) \in \mathcal{O}_{S}[{\bf x}]$ denote a sequence of semi-$q$-decomposable forms of degree $\ell$. Assume that $F_{n}=Q_{1, n} \cdots Q_{q, n}$ over $k'$ with $\deg  Q_{j, n}=d_{j}, 1 \leqslant j \leqslant q$ and $\{Q_{1, n}, \ldots, Q_{q, n}\}$ has the distributive constant not exceeding a positive number $\Delta$ for each $n$. Let $d=\max _{1 \leqslant j \leqslant q} d_{j}$. Assume that $\ell>d\Delta\left(\frac{m}{2}+1\right)^{2}$. Let $G_{n}({\bf x}) \in \mathcal{O}_{S}[{\bf x}]$ such that $\deg  G_{n}<\ell-d\Delta\left(\frac{m}{2}+1\right)^{2}$ for each $n$. Then there does not exist an infinite sequence of $\mathcal{O}_{S}^{*}$-non-proportional ${\bf x}_{n} \in \mathcal{O}_{S}^{m+1}$ for which
\begin{align*}
F_{n}({\bf x}_{n})&=G_{n}({\bf x}_{n}) \neq 0, n=1,2, \ldots, \\
\log \prod_{v \in S}\|G_{n}\|_{v}&=o(\log H_{S}({\bf x}_{n}))\text { as } n\rightarrow \infty,\\
\text{and }h(F_{n})&=o(h({\bf x}_{n})) \text { as } n\rightarrow \infty.
\end{align*}
\end{theorem}
In particular, for each nonzero $S$-integer $b$, the equation $F(x_{0}, \ldots, x_{m})=b$ (under the given assumption in Corollary \ref{1.4} for $F$) has finitely many $\mathcal{O}_{S}^{*}$-nonproportional solutions, provided $\deg  F>d\Delta\left(\frac{m}{2}+1\right)^{2}$ with $d=\max _{1 \leqslant j \leqslant q} d_{j}$.
\section*{Disclosure statement}
No potential conflict of interest was reported by the author(s).

No funding was received for this research.

\vskip0.2cm
{\footnotesize 
\noindent
{\sc Si Duc Quang}\\
$^1$Department of Mathematics, Hanoi National University of Education,
136-Xuan Thuy, Cau Giay, Hanoi, Vietnam;
$^2$Institute of Natural Sciences, Hanoi National University of Education,
136-Xuan Thuy, Cau Giay, Hanoi, Vietnam.\\
\textit{E-mail}: quangsd@hnue.edu.vn

\end{document}